\documentclass[a4paper,reqno]{amsart}

\usepackage{amssymb}
\usepackage{amstext}
\usepackage{amsmath}
\usepackage{amscd}
\usepackage{amsthm}
\usepackage{amsfonts}
\usepackage{enumerate}
\usepackage{graphicx}
\usepackage{latexsym}
\usepackage{mathrsfs}
\usepackage{mathtools}
\usepackage[all]{xy}
\xyoption{all}

\usepackage{pstricks}
\usepackage{lscape}
\usepackage{comment}

\newtheorem{theorem}{Theorem}[section]

\newtheorem{corollary}[theorem]{Corollary}

\newtheorem{proposition}[theorem]{Proposition}
\newtheorem{definition-proposition}[theorem]{Definition-Proposition}

\newtheorem{question}[theorem]{Question}
\newtheorem{conjecture}[theorem]{Conjecture}

\theoremstyle{definition}
\newtheorem{definition}[theorem]{Definition}

\newtheorem{remark}[theorem]{Remark}
\newtheorem{example}[theorem]{Example}

\newcommand{\Hom}{\operatorname{Hom}\nolimits}

\newcommand{\Tr}{\operatorname{Tr}\nolimits}
\renewcommand{\mod}{\mathsf{mod}\hspace{.01in}}

\newcommand{\add}{\mathsf{add}\hspace{.01in}}

\newcommand{\RHom}{\mathbf{R}\strut\kern-.2em\operatorname{Hom}\nolimits}

\numberwithin{equation}{section}

\usepackage{paralist}

\def\Im{\mathop{\rm Im}\nolimits}

\def\Tr{\mathop{\rm Tr}\nolimits}

\def\id{\mathop{\rm id}\nolimits}
\def\pd{\mathop{\rm pd}\nolimits}
\def\add{\mathop{\rm add}\nolimits}

\begin{document}
\title{Applications of Gorenstein projective $\tau$-rigid modules }
 \thanks{$*$ is the corresponding author.}
\author{Hui Liu}
\address{H.Liu:  Department of Mathematics, Nanjing University, Nanjing, 210093, P. R. China.}
\email{liuhui@smail.nju.edu.cn}
\author{Xiaojin Zhang}
\address{X. Zhang:  School of Mathematics and Statistics, Jiangsu Normal University, Xuzhou, 221116, P. R. China.}
\email{xjzhang@jsnu.edu.cn, xjzhangmaths@163.com}
\author{Yingying Zhang$^*$}
\address{Y. Zhang: Department of Mathematics, Huzhou University, Huzhou 313000, Zhejiang Province, P.R.China}
\email{yyzhang@zjhu.edu.cn}
\maketitle
\begin{abstract} We first introduce the notion of $CM$-$\tau$-tilting free algebras as the generalization of $CM$-free algebras and show the homological properties of $CM$-$\tau$-tilting free algebras.  Then we give a bijection between Gorenstein projective $\tau$-rigid modules and certain modules by using an equivalence established by Kong and Zhang. Finally, we give a partial answer to Tachikawa's first conjecture by using Gorenstein projective $\tau$-rigid modules.
\end{abstract}

\section{Introduction}

In 2014, Adachi, Iyama and Reiten \cite{AIR} introduced $\tau$-tilting theory as a generalization of tilting theory from the viewpoint of mutation. They showed that $\tau$-tilting theory is closely related to silting theory \cite{AI} and cluster tilting theory \cite{IY}. (support) $\tau$-tilting modules are the most important objects in $\tau$-tilting theory. So it is interesting to study (support) $\tau$-tilting modules for given algebras. For the recent development on this topics, we refer to \cite{AH, FGLZ, IZ1, IZ2, KK, PMH, W, WZ, Z, Zh, Zi}.

On the other hand, Gorenstein projective modules which form the main body of Gorenstein homological algebra can be traced back to Auslander-Bridger's modules of $G$-dimension zero \cite{AuB}. Gorenstein-projective modules are particular important to the Gorenstein algebras \cite{EJ2}. The definition of Gorenstein projective modules over an arbitrary ring was given by Enochs and Jenda \cite{EJ1, EJ2}. From then on, Gorenstein projective modules have gained a lot of attention in both homological algebra and the representation theory of finite-dimensional algebras. In this paper, we focus on the finitely generated Gorenstein projective modules over finite dimensional algebras over an algebraically closed field. For the recent development on this topics, we refer to \cite{CSZ, HLXZ, RZ1, RZ2, RZ3, RZ4}.

CM-finite algebras are important in both commutative algebra and representation theory of finite dimensional algebras which were studied in \cite{B, C, LZ2}. In \cite{XZ}, Xie and Zhang introduced the notion of Gorenstein projective $\tau$-rigid modules and studied the homological properties of Gorenstein projective $\tau$-rigid modules. Moreover, they introduced the notion of CM-$\tau$-tilting finite algebras as the generalization of both CM-finite algebras and $\tau$-tilting finite algebras introduced in \cite{DIJ}. We refer to \cite{LZh2} for the existence of non-projective Gorenstein projective $\tau$-tilting modules. We refer to \cite{GMZ} for the theory of $2$-term Gorenstein silting complexes. Note that CM-free algebras are CM-finite algebras. It is natural to define the class of CM-$\tau$-tilting free algebras as a special case of the class of  CM-$\tau$-tilting finite algebras. From now on, let $\Lambda$ be a finite dimensional algebra over an algebraically closed field. We call $\Lambda$ CM-$\tau$-tilting free if all finitely generated Gorenstein projective $\tau$-rigid left $\Lambda$-modules are projective modules. The following theorem is the first main result in this paper.

\begin{theorem}\label{0.1}
\begin{enumerate}[\rm(1)] Let $\Lambda$ be an algebra. Then we have the following
\item $\Lambda$ is CM-$\tau$-tilting free if and only if the opposite algebra $\Lambda^{o}$ is CM-$\tau$-tilting free.
\item A Gorenstein algebra $\Lambda$ is CM-$\tau$-tilting free if and only if so is the lower triangular matrix algebra $T_2(\Lambda)$.
\end{enumerate}
\end{theorem}

 For an algebra $\Lambda$ and $E$ a generator in the category of finitely generated left $\Lambda$-modules. In \cite{KZ}, Kong and Zhang defined $E$-Gorenstein projective modules and showed that there is an equivalence between the set of finitely generated $E$-Gorenstein projective modules and the set of Gorenstein projective modules over $\Gamma=({\rm End}E)^{op}$. As a result, they showed that the CM-Auslander algebra of a CM-finite algebra is indeed CM-free. In this paper, we use the equivalence defined by Kong and Zhang to study Gorenstein projective $\tau$-rigid modules.

We introduce $E$-rigid modules as the generalization of $\tau$-rigid modules. For an algebra $\Lambda$, we denote by $\Lambda$-$\mod$ the category of finitely generated left $\Lambda$-modules. Let $E$ be a generator in $\Lambda$-$\mod$. Let $M\in\Lambda$-$\mod$ and $E_1\stackrel{f_1}{\rightarrow} E_0\stackrel{f_1}{\rightarrow} M\rightarrow0$ an exact sequence with $f_i$ the minimal right  ${\rm add}E$-approximation of ${\rm Im}f_i$. We call $M$ $E$-rigid if ${\rm Hom}_{\Lambda}(f_1,M)$ is surjective. Putting $E=\Lambda$, one gets $\tau$-rigid modules.  We call a module $M$ $E$-Gorenstein projective $E$-rigid if it is both $E$-Gorenstein projective and $E$-rigid.

Denote by $\mathcal {EG}(E)$ the set of isomorphism classes of indecomposable $E$-Gorenstein projecitve $E$-rigid modules in $\Lambda$-$\mod$ and by $\mathcal{TG}(\Gamma)$ the set of isomorphism classes of indecomposable Gorenstein projective $\tau$-rigid modules in $\Gamma$-$\mod$. Now we are in a position to state our second main result in this paper.

\begin{theorem}\label{0.2} Let $E$ be a generator in $\Lambda$-$\mod$ and $\Gamma=({\rm End}_{\Lambda}E)^{op}$. Then there is a bijection between $\mathcal {EG}(E)$ and $\mathcal{TG}(\Gamma)$.
\end{theorem}

Let $E$ be a generator in $\Lambda$-$\mod$. We call $\Lambda$ {\it CM-$E$-free} if every indecomposable $E$-Gorenstein projective $E$-rigid module is a direct summand of $E$.
Taking $E=\Lambda$, the CM-$E$-free algebras are precisely CM-$\tau$-tilting free algebras above.

As an application, we have the following corollary.

\begin{corollary}\label{0.3} Let $E$ be a generator in $\Lambda$-$\mod$. Then $\Lambda$ is CM-$E$-free if and only if $\Gamma=({\rm End}E)^{op}$ is CM-$\tau$-tilting free.
\end{corollary}

Homological conjectures are hot in both homological algebra and representation theory of finite dimensional algebras. The most famous one of these conjectures is the Nakayama Conjecture which states that for an algebra $\Lambda$ if every term of its minimal injective resolution is projective, then $\Lambda$ is self-injective. To study the Nakayama conjecture, Tachikawa \cite{T} divided Nakayama conjecture into the following two conjectures which are well-known as Tachikawa's first conjecture and Tachikawa's second conjecture.

\begin{conjecture} Let $\Lambda$ be an algebra.
\begin{enumerate}[\rm(1)]
\item If ${\rm Ext}_{\Lambda}^i(\mathbb{D}\Lambda,\Lambda)=0$ for $i\geq1$, then $\Lambda$ is a self-injective algebra.
\item Let $\Lambda$ be a self-injective algebra and $M\in \Lambda$-$\mod$. If ${\rm Ext}_{\Lambda}^{i}(M,M)=0$ for $i\geq1$, then $M$ is projective.
\end{enumerate}
\end{conjecture}

The following is the third main result in this paper, which gives a partial answer to Tachikawa's first conjecture.

\begin{theorem}\label{0.4}  Let $\Lambda$ be an algebra. Then the following are equivalent.
\begin{enumerate}[\rm(1)]
\item $\Lambda$ is self-injective.
\item $\mathbb{D}\Lambda$ is semi-Gorenstein projective $\tau$-rigid.
\item Every cotilting module is semi-Gorenstein projective $\tau$-rigid.
\item $\mathbb{D}\Lambda$ is Gorenstein projective.
\item Every coltilting module is Gorenstein projective.
\item $\Lambda$ is semi-Gorenstein injective $\tau^{-1}$-rigid.
\item Every tilting module is semi-Gorenstein injective $\tau^{-1}$-rigid.
\item $\Lambda$ is Gorenstein-injective.
\item  Every tilting module is Gorenstein injective.

\end{enumerate}
\end{theorem}

Throughout this paper, we assume that all algebras are finite dimensional algebras over an algebraically closed field and all modules are finitely generated left modules.
Denote by $\tau$ (resp. $\tau^{-1}$) the Auslander-Reiten translation functor (resp. the inverse of the Auslander-Reiten translation functor.)

\section{Preliminaries}

In this section, we recall some preliminaries on Gorenstein projecitve modules, $\tau$-tilting modules and lower triangular matrix algebras. Let $K$ be an algebraically closed field. We denote by $\mathbb{D} := \Hom_{K}(-, K)$ the K-duality. Let $\Lambda$ be a finite dimensional algebra over an algebraically closed field $K$. We denote by $\Lambda$-$\mod$ the category of finitely generated left A-modules. Denote by $\mathcal{P}(\Lambda)$ (resp. $\mathcal{I}(\Lambda)$) the category of finitely generated projective (resp. injective) left $\Lambda$-modules.

\subsection{Gorenstein projecitve modules}
Firstly, we need the definition of Gorenstein projective modules and the definition of Gorenstein injective modules in \cite{EJ1,EJ2}.
\begin{definition}\label{2.1} Let $\Lambda$ be an algebra and $M\in\Lambda$-$\mod$.
\begin{enumerate}[\rm(1)]
\item  $M$ is called {\it Gorenstein projective}, if there is an exact sequence $\cdots \rightarrow P_{-1} \rightarrow P_{0} \rightarrow P_{1} \rightarrow \cdots$ in $\mathcal{P}(\Lambda)$, which stays exact under $\Hom_{\Lambda}(-,\Lambda)=(-)^{\ast}$, such that $M\simeq\Im ( P_{-1} \rightarrow P_{0})$.
\item $M$ is called {\it Gorenstein injective}, if there is an exact sequence $\cdots \rightarrow I_{-1} \rightarrow I_{0} \rightarrow I_{1} \rightarrow \cdots$ in $\mathcal{I}(\Lambda)$, which stays exact under $\Hom_{\Lambda}(\mathbb{D}\Lambda,-)$, such that $M\simeq\Im ( I_{-1} \rightarrow I_{0})$.
\end{enumerate}
\end{definition}

Denote by $\mathcal{GP}(\Lambda)$ the category of all finitely generated Gorenstein projective left $\Lambda$-modules and by $\underline{\mathcal{GP}(\Lambda)}$ the stable category of $\mathcal{GP}(\Lambda)$ which modulo projective modules. The following properties of Gorenstein projective modules in \cite{AuR} are well-known.

\begin{proposition}\label{2.2}
Let $\Lambda$ be an algebra.
 \begin{enumerate}[\rm(1)]
\item $(-)^*=\Hom_{\Lambda}(-,\Lambda): \mathcal{GP}(\Lambda) \rightarrow \mathcal{GP}(\Lambda^{\mathrm{op}})$ is a duality.

\item $\Omega^{1} : \underline{\mathcal{GP}(\Lambda)} \rightarrow \underline{\mathcal{GP}(\Lambda)}$ is an equivalence of categories.

\item $\Tr : \underline{\mathcal{GP}(\Lambda)} \rightarrow \underline{\mathcal{GP}(\Lambda^{\mathrm{op}})}$ is a duality.
\end{enumerate}
\end{proposition}

We also need the following property of tensor product of Gorenstein projective modules \cite[Proposition 2.6]{HLXZ}.

\begin{proposition}\label{2.3} Let $A$ and $B$ be two algebras. Let $M\in A$-$\mod$ and $N\in B$-$\mod$
 be Gorenstein projective modules. Then $M \otimes N \in (A \otimes B)$-$\mod$ is
Gorenstein projective.
\end{proposition}

\subsection{$\tau$-tilting modules}
Let $\Lambda$ be an algebra. For a module $M\in \Lambda$-$\mod$, denote by $|M|$ the number of pairwise
non-isomorphic indecomposable summands of $M$. Now we recall the definitions of $\tau$-rigid  modules (resp. $\tau^{-1}$-rigid) and $\tau$-tilting modules (resp. $\tau^{-1}$-tilting) from \cite{S} and \cite{AIR}.

\begin{definition}\label{2.4}Let $\Lambda$ be an algebra and $M\in \Lambda$-$\mod$.
\begin{enumerate}[\rm(1)]
\item We call $M$ {\it $\tau$-rigid} if $\Hom_{\Lambda}(M,\tau M)=0$. Moreover, $M$ is called a {\it $\tau$-tilting} module if $M$ is $\tau$-rigid
and $|M|=|\Lambda|$.
\item We call $M$ {\it support $\tau$-tilting} if there exists an idempotent $e$ of $\Lambda$ such that $M$ is a $\tau$-tilting $\Lambda/(e)$-module.

\item We call $M $ {\it $\tau^{-1}$-rigid} if $\Hom_{\Lambda}(\tau^{-1}M, M)=0$.
 Moreover, $M $ is called a {\it $\tau^{-1}$-tilting} module if $M$ is $\tau^{-1}$-rigid
and $|M|=|\Lambda|$.
\item We call $M$ {\it support $\tau^{-1}$-tilting} if $M$ is a $\tau^{-1}$-tilting $\Lambda/(e)$-module for some idempotent $e$ of $\Lambda$.
\end{enumerate}
\end{definition}

The following result in \cite[Proposition 2.4(c)]{AIR} is very essential in this paper.

\begin{proposition}\label{2.5}
 Let $M\in\Lambda$-$\mod $ and $P_1(M)\stackrel{f}{\rightarrow}P_0(M)\rightarrow M\rightarrow 0$ a minimal projective presentation of $M$.
 Then $M$ is $\tau$-rigid if and only if $\Hom_{\Lambda }(f,M)$ is epic.
 \end{proposition}

 \subsection{Lower triangular matrix algebras}

 Let $A$ and $B$ be rings and let ${_B}M_A$ be a $B$-$A$-bimodule. Then the lower triangular matrix ring $\left(\begin{smallmatrix}
	A&0\\
	M&B
\end{smallmatrix}\right)$ is the ring with addition $\left(\begin{smallmatrix}
r&0\\
m&s
\end{smallmatrix}\right)+\left(\begin{smallmatrix}
r'&0\\
m'&s'
\end{smallmatrix}\right)=\left(\begin{smallmatrix}
r+r'&0\\
m+m'&s+s'
\end{smallmatrix}\right)$ and product $\left(\begin{smallmatrix}
r&0\\
m&s
\end{smallmatrix}\right)\left(\begin{smallmatrix}
r'&0\\
m'&s'
\end{smallmatrix}\right)=\left(\begin{smallmatrix}
rr'&0\\
mr'+sm'&ss'
\end{smallmatrix}\right).$  A left $\left(\begin{smallmatrix}
	A&0\\
	M&B
\end{smallmatrix}\right)$-module $(X,Y,f)$ is decided uniquely by a morphism $f:M\otimes_AX\rightarrow Y$ of left $B$-modules with $X\in A$-mod and $Y\in B$-mod. Throughout this paper, we assume that both $A$ and $B$ are finite-dimensional algebras over an algebraically closed field $K$ and $M$ is finitely generated in both $B$-mod and $A^{o}$-mod. Putting $A=B=M$, one gets the lower triangular matrix algebra $T_2(A)$.

We also need the following facts on the projective modules over $\left(\begin{smallmatrix}
	A&0\\
	M&B
\end{smallmatrix}\right)$ \cite[III Proposition 2.5]{AuRS}.

\begin{proposition}Let $\Lambda=\left(\begin{smallmatrix}
	A&0\\
	M&B
\end{smallmatrix}\right)$ be a lower triangular matrix algebra. Then the indecomposable projective modules in $\Lambda$-mod are of the form $(P,M\otimes P, I_{M\otimes P})$ and of the form $(0,Q,0)$, where $P$ is indecomposable projective in $A$-mod and $Q$ is indecomposable projective in $B$-mod.
\end{proposition}

\section{CM-$\tau$-tilting free algebras}

In this section, we first introduce CM-$\tau$-tilting free algebras which are the generalization of CM-free algebras and local algebras. Then we study the homological properties of these algebras and give a construction of CM-$\tau$-tilting free algebras by using $T_2$-extension.

We start with the definition of Gorenstein projective $\tau$-rigid modules.

\begin{definition}\label{4.b}  Let $\Lambda$ be an algebra and $M\in\Lambda$-$\mod$. $M$ is called {\it Gorenstein projective $\tau$-rigid} if it is both Gorenstein projective and $\tau$-rigid.
\end{definition}

Recall from \cite{XZ} that an algebra $\Lambda$ is called CM-$\tau$-tilting finite if there are only finite indecomposable Gorenstein projective $\tau$-rigid modules in $\Lambda$-$\mod$ up to isomorphism. In the following we introduce a special class of CM-$\tau$-tilting finite  algebras called CM-$\tau$-free algebras.

\begin{definition}\label{4.1} Let $\Lambda$ be an algebra. We call $\Lambda$ CM-$\tau$-tilting free if all Gorenstein projective $\tau$-rigid modules are projective.
\end{definition}

It is obvious that CM-free algebras are CM-$\tau$-tilting free. While the converse is not true. In the following we give an example due to Ringel and Zhang [RZ1, Example 6.1] to show this.

\begin{example}\label{4.2} Let $k$ be a infinite field and $q \in k$ and $q\neq0$. Let $\Lambda$ be a $6$-dimensional local algebra $\Lambda(q)$.
\begin{enumerate}[\rm(1)]
\item There are infinite many Gorenstein projective modules of dimension $3$ and hence $\Lambda$ is not CM-free.
\item All $\tau$-rigid modules in $\mod \Lambda$ are projective since $\Lambda$ is local. Therefore, $\Lambda$ is CM-$\tau$-tilting free.
\end{enumerate}
\end{example}

In the following we show the symmetric properties of CM-$\tau$-tilting free algebras.

\begin{proposition}\label{4.3} Let $\Lambda$ be an algebra. $\Lambda$ is CM-$\tau$-tilting free if and only if $\Lambda^{o}$ is CM-$\tau$-tilting free.
\end{proposition}

\begin{proof}  By \cite[Theorem 2.14]{AIR}, one gets that  for any $\tau$-rigid module $M$ without projective direct summands, $M$ is $\tau$-rigid if and only if so is ${\rm Tr}M$.

We only show the necessity since the sufficiency is similar. Suppose that $\Lambda$ is CM-$\tau$-tilting free. On the contrary, suppose that $\Lambda^{o}$ is not CM-$\tau$-tilting free. Then there is a indecomposable Gorenstein projective $\tau$-rigid module $N \in\Lambda^{o}$-$\mod$, then one gets ${\rm Tr}N \in\Lambda$-$\mod$ is a Gorenstein projective $\tau$-rigid module which is not projective by Proposition \ref{2.2}, a contradiction.
\end{proof}

An algebra $\Lambda$ is Gorenstein if $\id _{\Lambda}\Lambda< \infty$ and $\id \Lambda_{\Lambda}< \infty$. A Gorenstein algebra $\Lambda$ is $n$-Gorenstein if $\id _{\Lambda}\Lambda\leq n$.
In the following, we classify Gorenstein CM-free algebras which are well-known.

\begin{proposition}\label{4.4} Let $\Lambda$ be an $n$-Gorenstein algebra. If $\Lambda$ is CM-free, then $\Lambda$ is of global dimension at most $n$.
\end{proposition}

\begin{proof} Since $\Lambda$ is $n$-Gorenstein, then for any $M\in\Lambda$-$\mod$, $\Omega^n M$ is Gorenstein projective by \cite[Theorem 4.1]{AuR}. Since $\Lambda$ is CM-free, then $\Omega^n M$ is projective, and hence $\pd M\leq n$. The assertion holds.
\end{proof}

It is natural to ask the following question.

\begin{question}\label{4.a} Is a Gorenstein CM-$\tau$-tilting free algebra local or of finite global dimension?
\end{question}

We give a negative answer to the Question \ref{4.a} by giving the following counter-example.

\begin{example}\label{4.5} Let $\Lambda$ be an algebra given by the quiver $$\xymatrix@C=1cm{1\ar@(ul,dl)_{\alpha }\ar[r]&2}$$ with the relation $\alpha^2=0$.
\begin{enumerate}[\rm(1)]
\item $\Lambda$ is a $1$-Gorenstein algebra with infinite global dimension which is not local.
\item  The support $\tau$-tilting quiver of $\Lambda$ is as follows:
$$\begin{xy}
(0,0)*+{\begin{smallmatrix} &1\\1&&2\\2&\end{smallmatrix}\oplus\begin{smallmatrix}2\end{smallmatrix}}="1",
(0,-15)*+{\begin{smallmatrix} 2\end{smallmatrix}}="2",
(22,0)*+{\begin{smallmatrix} &1\\1&&2\\2&\end{smallmatrix}\oplus\begin{smallmatrix}1\\ 1\\2\end{smallmatrix}}="3",
(44,0)*+{\begin{smallmatrix} 1\\1\end{smallmatrix}\oplus\begin{smallmatrix}1\\ 1\\2\end{smallmatrix}}="4",
(44,-15)*+{\begin{smallmatrix} 1\\1\end{smallmatrix}}="5",
(22,-15)*+{\begin{smallmatrix} 0\end{smallmatrix}}="6",
\ar"1";"2",\ar"1";"3",\ar"3";"4",\ar"2";"6",
\ar"4";"5",\ar"5";"6",
\end{xy}$$
And hence $\Lambda$ is a CM-$\tau$-tilting free algebra.
\end{enumerate}
\end{example}

In the following we try to construct more examples on CM-$\tau$-tilting free algebras. We need the following results in \cite[Theorem 1.1]{LZ1} and \cite[Corollary 3.4]{LZh}

\begin{proposition}\label{4.6} Let $\Lambda$ be a Gorenstein algebra.
\begin{enumerate}[\rm(1)]

 \item A $T_2(\Lambda)$-module $(X,Y,\phi)$ with $\phi:X\rightarrow Y$ is Gorenstein-projective if and only if $X,
Y$ and ${\rm Coker}\phi$ are Gorenstein-projective modules in $\Lambda$-$\mod$ and $\phi : X\rightarrow Y$ is injective.

\item If a $T_2(\Lambda)$-module $(X,Y,\phi)$ with $\phi:X\rightarrow Y$ is $\tau$-rigid, then $X$ is $\tau$-rigid and ${\rm Coker \phi}$  is $\tau$-rigid.

\end{enumerate}
\end{proposition}

We also need the following property on the tensor product of $\tau$-rigid modules \cite[Proposition 3.3]{LZh2}.

 \begin{proposition}\label{2.6} Let $A$ and $B$ be two algebras. If $P\in A$-$\mod$ is a projective module and $M\in B$-$\mod$ is a $\tau$-rigid module, then $P\otimes M\in (A\otimes B)$-$\mod$ is a $\tau$-rigid module.
\end{proposition}

Now we are in a position to show the $T_2$-extension of a Gorenstein CM-$\tau$-tilting free algebra $\Lambda$. As a result, we can construct a class of Gorenstein CM-$\tau$-tilting free algebras.

\begin{theorem}\label{4.5} Let $\Lambda$ be a Gorenstein algebra. Then the lower triangular matrix algebra $T_2(\Lambda)$ is a CM-$\tau$-tilting free algebra if and only if  $\Lambda$ is a CM-$\tau$-tilting free algebra.
\end{theorem}

\begin{proof} $\Leftarrow$ For a $T_2(\Lambda)$-module $(X,Y,\phi)$, assume that $(X,Y,\phi)$ is Gorenstein projective $\tau$-rigid. By Proposition \ref{4.6}(1), $X,
Y$ and ${\rm Coker}\phi$ are Gorenstein-projective modules in $\Lambda$-$\mod$ and $\phi : X\rightarrow Y$ is injective. Then by Proposition \ref{4.6}(2), $X$ is $\tau$-rigid and ${\rm Coker \phi}$  is $\tau$-rigid. Since $\Lambda$ is CM-$\tau$-tilting free, then the fact both $X$ and ${\rm Coker} \phi$ are Gorenstein projective $\tau$-rigid implies that both $X$ and ${\rm Coker} \phi$ are projective. Since $0\rightarrow X\rightarrow Y\rightarrow {\rm Coker}\phi\rightarrow 0$ is an exact sequence, one gets that the sequence splits, and hence $(X,Y,\phi)$ is a direct sum of $(X,X, I_X)$ and $(0, {\rm Coker}\phi,0)$. Then one gets the assertion.

$\Rightarrow$ On the contrary, suppose that $\Lambda$ is not CM-$\tau$-tilting free. Then we get a non-projective Gorenstein projective $\tau$-rigid module $M\in \Lambda$-$\mod$.  Since $T_2(\Lambda)\simeq T_2(K)\otimes\Lambda$, we get $T_2(K)\otimes M$ is Gorenstein projective and $\tau$-rigid by Proposition \ref{2.3} and Proposition \ref{2.6}. Since $M$ is non-projective, one gets that $T_2(K)\otimes M$ is non-projective by \cite[Lemma 2.7]{LZh2}. Note that $T_2(\Lambda)$ is CM-$\tau$-tilting free,  a contradiction.
\end{proof}

As a straight application, we get the following corollary.

\begin{corollary}\label{4.7} Let $\Lambda$ be an algebra. Then $\Lambda$ is Gorenstein CM-$\tau$-tilting free if and only if so is the lower triangular matrix algebra $T_2(\Lambda)$.
\end{corollary}

\begin{proof} By Theorem \ref{4.5} it suffices to show that $\Lambda$ is Gorenstein if and only if so ia $T_2(\Lambda)$. This follows from the facts $T_2(\Lambda)\simeq T_2(K)\otimes \Lambda$ and $\rm {id}_{\Lambda} T_2(\Lambda)={\rm id}_{\Lambda}\Lambda+1$.
\end{proof}

\section{$E$-rigid modules}

In this section, we introduce the notion of $E$-Gorenstein projective and $E$-rigid modules, and show the bijection between Gorenstein projective $\tau$-rigid modules and $E$-Gorenstein projective $E$-rigid modules. Let $\Lambda$ be an algebra and $E$ a generator in $\Lambda$-$\mod$. In the following we recall the definition of $E$-Gorenstein projective modules from \cite{KZ}.

\begin{definition}\label{5.1} Let $\dots\rightarrow E_{-1}\stackrel{d_{-1}}{\rightarrow} E_0\stackrel{d_0}{\rightarrow}E_1\rightarrow\dots\ \ \ \ \ (*)$ be an exact sequence with $E_i\in\add E$ which are still exact by applying the functors ${\rm Hom}_{\Lambda}(-,E)$ and ${\rm Hom}_{\Lambda}(E,-)$. Then we call $M\simeq {\rm Im d_0}$ an $E$-Gorenstein projective module. Denote by $\mathcal{G}(E)$ the subcategory of $\Lambda$-$\mod$ consisting of all $E$-Gorenstein projective modules.
\end{definition}

For an algebra $\Gamma$, denote by $\mathcal{G}(\Gamma)$ the category of finitely generated Gorenstein projective left $\Lambda$-modules. The following property of $E$-Gorenstein projective modules is essential \cite[Theorem 4.3]{KZ}.

\begin{proposition}\label{5.2}
Let $\Lambda$ be an algebra and $E$ a generator in $\Lambda$-$\mod$, and $\Gamma=({\rm End}_\Lambda E)^{op}$. Then ${\rm Hom}_{\Lambda}(E, -)$ induces an equivalence between $\mathcal{G}(E)$and $\mathcal{G}(\Gamma)$.
\end{proposition}

In the following, we introduce the notion of $E$-rigid modules which are the generalization of $\tau$-rigid modules.

\begin{definition}\label{5.4} Let $E$ be a generator in $\Lambda$-$\mod$ and $M\in\Lambda$-$\mod$. Let $E_1\stackrel{f_1}{\rightarrow}E_0\stackrel{f_0}{\rightarrow} M\rightarrow 0$ be an exact sequence with $f_i$ the minimal right $\add E$-approximation of $\Im f_i$ {(that is, a minimal $\add E$-presentation of $M$)}. If the map ${\rm Hom}_{\Lambda}(f_1,M) $ is surjective, then we call $M$ an $E$-rigid module. In particular, if $E=\Lambda$, then $E$-rigid modules are precisely $\tau$-rigid modules.
\end{definition}

We give the following example on $E$-rigid modules.

\begin{example}\label{5.5} Let $\Lambda=KQ$ be given by the quiver $Q:1\rightarrow 2\rightarrow 3$. Let $E_1=P(1)\oplus P(2)\oplus P(3)\oplus I(2)$ be a generator in $\Lambda$-$\mod$.
\begin{enumerate}[\rm(1)]
\item $E_1$ is $E_1$-rigid but not rigid and hence not $\tau$-rigid since we have the following non-split exact sequence: $$0\rightarrow P(3)\rightarrow P(1)\rightarrow I(2)\rightarrow0.$$
\item The following is a minimal $\add E_1$-presentation of $S(1):$ $$ P(2)\rightarrow I(2)\rightarrow S(1)\rightarrow 0.$$ Hence $S(1)$ is $E_1$-rigid but not in $\add E_1$.
\end{enumerate}
\end{example}

In the following we show the property on $E$-rigid modules which generalizes \cite [Proposition 2.5] {AIR}.
\begin{proposition}\label{5.6} Let $M, E\in$ $\Lambda$-$\mod$ and $E$ a generator. Let $E_1 \stackrel{f_1}{\rightarrow} E_0\stackrel{f_0}{\rightarrow}M\rightarrow 0$ be a minimal $\add E$-presentation of $M$. If $M$ is $E$-rigid, then $E_0$ and $E_1$ have no non-zero direct summands in common.
\end{proposition}

\begin{proof} We show that any homomorphism $s: E_1\rightarrow E_0$ is in the radical. One gets a homomorphism $f_0s: E_1\rightarrow M$. Since $M$ is $E$-rigid, by Proposition \ref{2.5} there exists a homomorphism $t: E_0\rightarrow M$ such that $tf_1=f_0s$. Since $f_0$ is a minimal right $\add E$-approximation of $M$, then there is a homomorphism $u:E_0\rightarrow E_0$ such that $t=f_0u$. So one gets $f_0(s-uf_1)=0$ which implies ${\rm Im}(s-uf_1)\in {\rm Im}f_1$. Since $E_1\rightarrow {\rm Im}f_1$ is a minimal right $\add E$-approximation, then there exists a homomorphism $v:E_1\rightarrow E_1$ such that $s=uf_1+f_1v$. We have the following commutative diagram:
$$\xymatrix{\ &E_1\ar[ld]_{v}\ar[d]^{s}\ar[r]^{f_1}&E_0\ar[d]^{t}\ar[dl]^{u}\ar[r]^{f_0}&M\ar[r]&0\\
E_1\ar[r]^{f_1}&E_0\ar[r]^{f_0}&M\ar[r]&0& \ }.$$ The assertion follows from the fact $f_1$ is in the radical.
\end{proof}

\begin{definition}\label{5.7} Let $E$ be a generator in $\Lambda$-$\mod$. We call $M\in\Lambda$-$\mod$ {\it $E$-Gorenstein projective $E$-rigid} if it is both $E$-Gorenstein projective and $E$-rigid.
\end{definition}

Denote by $\mathcal {EG}(E)$ the set of isomorphism classes of indecomposable $E$-Gorenstein projecitve $E$-rigid modules in $\Lambda$-$\mod$ and by $\mathcal{TG}(\Gamma)$ the set of isomorphism classes of indecomposable Gorenstein projective $\tau$-rigid modules in $\Gamma$-$\mod$. Now we are in a position to state our main results in this section.

\begin{theorem}\label{5.8} Let $E$ be a generator in $\Lambda$-$\mod$ and $\Gamma=({\rm End}_{\Lambda}E)^{op}$. Then there is a bijection between $\mathcal {EG}(E)$ and $\mathcal{TG}(\Gamma)$.
\end{theorem}

\begin{proof} We show that the functor ${\rm Hom}_{\Lambda}(E,-)$ is the desired map. We divide the proof into several steps.

\vspace{0.2cm}

(a) ${\rm Hom}_{\Lambda}(E,-)$ is a well-defined map from $\mathcal {EG}(E)$ to $\mathcal{TG}(\Gamma)$.

\vspace{0.2cm}

By Proposition \ref{5.2}, one gets that ${\rm Hom}_{\Lambda}(E,M)$ is indecomposable in $\mathcal{G}(\Gamma)$ for any indecomposable $M\in \mathcal{G}(E)$. For any $M\in\mathcal{EG}(E)$, we get a minimal $\add E$ presentation $$E_1\stackrel{f_1}{\rightarrow} E_0\stackrel{f_0}{\rightarrow} M\rightarrow 0.\ \ \ \ \ \ \ (1)$$

 Applying the functor ${\rm Hom}_{\Lambda}(E,-)$ to $(1)$, we get the following minimal projective presentation of ${\rm Hom}(E,M)$:
 $${\rm Hom}_{\Lambda}(E,E_1)\stackrel{{^*f_1}}{\rightarrow} {\rm Hom}_{\Lambda}(E,E_0)\stackrel{{^*f_0}}{\rightarrow} {\rm Hom}_{\Lambda}(E,M)\rightarrow0, \ \ \ \ \ \ \ \ (2)$$
 \noindent where $^*f_i={\rm Hom}_{\Lambda}(E,f_i)$ for $i=0,1$.

 Note that $M\in\mathcal{EG}(E)$, applying the functor ${\rm Hom}_{\Lambda}(-,M)$ to $(1)$, one gets the following exact sequence $${\rm Hom}_{\Lambda}(E_0, M)\stackrel{{f_1}'}{\rightarrow} {\rm Hom}_{\Lambda}(E_1, M)\rightarrow 0. $$ Applying the functor ${\rm Hom}_{\Gamma}(-,{\rm Hom}_{\Lambda}(E,M))$ to $^*f_1$, one gets the following complex $${\rm Hom}_{\Gamma}({\rm Hom}_{\Lambda}(E, E_0), {\rm Hom}_{\Lambda}(E,M))\stackrel{h}{\rightarrow} {\rm Hom}_{\Gamma}({\rm Hom}_{\Lambda}(E, E_1), {\rm Hom}_{\Lambda}(E,M)). $$
 By Proposition \ref{5.2}, one gets the following commutative diagram:
 $$\xymatrix{{\rm Hom}_{\Lambda}(E_0, M)\ar[r]^{{f_1}'}\ar[d]^{\cong}& {\rm Hom}_{\Lambda}(E_1, M)\ar[r]\ar[d]^{\cong}& 0\\
 {\rm Hom}_{\Gamma}({\rm Hom}_{\Lambda}(E, E_0), {\rm Hom}_{\Lambda}(E,M))\ar[r]^h&{\rm Hom}_{\Gamma}({\rm Hom}_{\Lambda}(E, E_1), {\rm Hom}_{\Lambda}(E,M))\ar[r]&0}$$
 which implies that $h$ is surjective. That is, ${\rm Hom}_{\Lambda}(E,M)$ is in $\mathcal {TG}(\Gamma)$ by Proposition \ref{2.5}.

 \vspace{0.2cm}

 (b) ${\rm Hom}_{\Lambda}(E,-)$ is injective.

 \vspace{0.2cm}

 Let $M_1, M_2$ be in $\mathcal{EG}(E)$ such that ${\rm Hom}_{\Lambda}(E,M_1)\simeq{\rm Hom}_{\Lambda}(E,M_2)$. By Proposition \ref{5.2}, one gets that $M_1\simeq M_2$.

 \vspace{0.2cm}

 (c) ${\rm Hom}_{\Lambda}(E,-)$ is surjective.

 \vspace{0.2cm}

 For any $N\in\mathcal{TG}(\Gamma)$, we get a minimal projective presentation of $N$: $$Q_1\stackrel{f_1}{\rightarrow} Q_0\stackrel{f_0}{\rightarrow} N\rightarrow 0.\ \ \ \ \ (3)$$
  By Proposition \ref{5.2}, we get the following minimal $\add E$-presentation of $M$:$$ E_1\stackrel{g_1}{\rightarrow} E_0\stackrel{g_0}\rightarrow M\rightarrow 0 \ \ \ \ \ \ (4)$$ such that $M\in\mathcal{G}(E)$, ${\rm Hom}_{\Lambda}(E,M)\simeq N$, ${\rm Hom}_{\Lambda}(E,E_i)\simeq Q_i$ and ${\rm Hom}_{\Lambda}(E,g_i)\simeq f_i$ for $i=0,1$.

  Note that $N\in\mathcal{TG}(\Gamma)$, applying the functor ${\rm Hom}_{\Lambda}(-,N)$ to $f_1$, one gets the following exact sequence $${\rm Hom}_{\Lambda}(Q_0, N)\stackrel{{f_1}^{'}}{\rightarrow} {\rm Hom}_{\Lambda}(Q_1, N)\rightarrow 0.$$  Similarly, applying the functor ${\rm Hom}_{\Lambda}(-,M)$ to $g_1$, one has the following complex $${\rm Hom}_{\Lambda}(E_0, M)\stackrel{{g_1}^{'}}{\rightarrow} {\rm Hom}_{\Lambda}(E_1, M).$$

  By using Proposition \ref{5.2} again, one gets the following commutative diagram:
  $$\xymatrix{{\rm Hom}_{\Gamma}(Q_0, N)\ar[r]^{{f_1}'}\ar[d]^{\cong}& {\rm Hom}_{\Gamma}(Q_1, N)\ar[r]\ar[d]^{\cong}& 0\\
  {\rm Hom}_{\Lambda}(E_0, M)\ar[r]^{{g_1}'}& {\rm Hom}_{\Lambda}(E_1, M)\ar[r]& 0}$$
 which implies that ${g_1}'$ is surjective, that is $M$ is $E$-rigid. And hence $M\in\mathcal{EG}(E)$.
 \end{proof}

 In the following we show the following application of Theorem \ref{5.8}.

\begin{definition}\label{5.9} Let $E$ be a generator in $\Lambda$-$\mod$.
\begin{enumerate}[\rm(1)]

\item We call $\Lambda$ {\it CM-E-free} if  every object in $\mathcal{EG}(E)$ is a direct summand of $E$.
Take $E=\Lambda$, the CM-$E$-free algebras are precisely CM-$\tau$-tilting free algebras in Section 3.

\item We call $\Lambda$ {\it CM-E-finite} if the number of objects in $\mathcal{EG}(E)$ is finite.
Take $E=\Lambda$, the CM-$E$-finite algebras are precisely CM-$\tau$-tilting finite algebras in Section 3.
\end{enumerate}
\end{definition}

As an application, we have the following corollary.

\begin{corollary}\label{5.10} Let $E$ be a generator in $\Lambda$-$\mod$. Then $\Lambda$ is CM-$E$-free if and only if $\Gamma=({\rm End}E)^{op}$ is CM-$\tau$-tilting free.
\end{corollary}
\begin{proof} It is a straight result of Theorem \ref{5.9}.
\end{proof}

We end this section with the following question(compare \cite[Theorem 4.5]{KZ}).

\begin{question}\label{5.11}  Let $E$ be a generator in $\Lambda$-$\mod$. Let $\Lambda$ be CM-E-finite with $F$ the additive generator of $E$-Gorenstein projective $E$-rigid modules. Is $\Gamma'=({\rm End}(F))^{op}$ a CM-$\tau$-tilting free algebra?
\end{question}

\section{Tachikawa's Conjecture}

In this section, we try to characterize self-injective algebras in terms of Gorenstein projective $\tau$-rigid modules. As a result, we can give a partial answer to Tachikawa's first conjecture.

Recall from \cite{RZ1} that a module $M\in \Lambda$-$\mod$ is called {\it semi-Gorenstein projective} if ${\rm Ext}_{\Lambda}^{i}(M,\Lambda)=0$ for any $i\geq1$. Dually,  we call a module $N\in$$\Lambda$-$\mod$ {\it semi-Gorenstein injective} if ${\rm Ext}_{\Lambda}^{i}(\mathbb{D}\Lambda, N)=0$ for $i\geq1$. The Tachikawa's first conjecture is equivalent to that for an algebra $\Lambda$, if $\mathbb{D}\Lambda$ is semi-Gorenstein projective, then $\Lambda$ is self-injective.

To show the main result in this section, we need the following result in \cite{XZZ}.

\begin{proposition}\label{3.1} Let $\Lambda$ be an algebra.  The following are equivalent.
\begin{enumerate}[\rm(1)]
\item $\Lambda$ is a $1$-Gorenstein algebra.
\item $\mathbb{D}\Lambda$ is $\tau$-rigid
\item $\Lambda$ is $\tau^{-1}$-rigid.
\end{enumerate}
\end{proposition}

Note that a self-injective algebra should be a $1$-Gorenstein algebra, by Proposition \ref{3.1} it is natural to study Tachikawa's first conjecture in terms of $\tau$-rigid modules.
In the following we give our main result in this section which is a partial answer to Tachikawa's first conjecture.

\begin{theorem}\label{3.2} Let $\Lambda$ be an algebra. Then the following are equivalent.
\begin{enumerate}[\rm(1)]
\item $\Lambda$ is self-injective.
\item $\mathbb{D}\Lambda$ is semi-Gorenstein projective $\tau$-rigid.
\item Every cotilting module is semi-Gorenstein projective $\tau$-rigid.
\item $\mathbb{D}\Lambda$ is Gorenstein projective.
\item Every coltilting module is Gorenstein projective.
\item $\Lambda$ is semi-Gorenstein-injective $\tau^{-1}$-rigid.
\item Every tilting module is semi-Gorenstein injective $\tau^{-1}$-rigid.
\item $\Lambda$ is Gorenstein-injective.
\item  Every tilting module is Gorenstein-injective.

\end{enumerate}
\end{theorem}

\begin{proof} (a) Since $(1)\Leftrightarrow(4)$ is from the definitions of Gorenstein projective modules and self-injective algebras, we only show $(1)\Leftrightarrow (2)$,$(2)\Leftrightarrow(3)$, and $(4)\Leftrightarrow(5)$.

$(1)\Rightarrow (2)$ Since $\Lambda$ is self-injective, then $\mathbb{D}\Lambda$ is projective and hence semi-Gorenstein projective and $\tau$-rigid.

$(2)\Rightarrow(1)$  Since $\mathbb{D}\Lambda$ is $\tau$-rigid and $|\mathbb{D}\Lambda|=|\Lambda|$, one gets that $\mathbb{D}\Lambda$ is a $\tau$-tilting module. Note that $\mathbb{D}\Lambda$ is faithful, then $\mathbb{D}\Lambda$ is a tilting module. Since $\mathbb{D}\Lambda$ is semi-Gorenstein projective, then
 the fact $\rm{pd}\mathbb{D}\Lambda\leq 1$ implies that $\mathbb{D}\Lambda$ is projective. The assertion holds.

$(3)\Rightarrow (2)$  This is obvious.

$(2)\Rightarrow (3)$  By $(2)\Leftrightarrow (1)$, $\Lambda$ is self-injective, and hence $\Lambda$ is the unique cotilting module in $\mod\Lambda$. The assertion holds.

$(5)\Rightarrow(4)$ This is obvious.

$(4)\Rightarrow(5)$  By (4) one gets that $\Lambda$ is self-injective, and hence $\Lambda$ is the unique cotilting module in $\mod\Lambda$.

(b) Since $(1)\Leftrightarrow (8)$ is clear, we only need to show $(1)\Leftrightarrow(6)$, $(6)\Leftrightarrow(7)$ and $(8)\Leftrightarrow(9)$.

$(1)\Rightarrow(6)$ Since $\Lambda$ is self-injective, then $\Lambda\simeq \mathbb{D}\Lambda$ is injective, and hence semi-Gorenstein injective and $\tau^{-1}$-rigid.

$(6)\Rightarrow (1)$ Since $\Lambda$ is $\tau^{-1}$-rigid, by Proposition \ref{3.1}, one gets that $\Lambda$ is $1$-Gorenstein. Note that $\Lambda$ is semi-Gorenstein injective, one gets ${\rm Ext}_{\Lambda}^{i}(\mathbb{D}\Lambda,\Lambda)=0$ for $i\geq1$, which implies that $\Lambda$ is self-injective.

$(7)\Rightarrow (6)$ This is obvious.

$(6)\Rightarrow (7)$ By $(6)\Leftrightarrow(1)$, $\Lambda$ is self-injective. Then the unique tilting module in $\Lambda$-$\mod$ is $\mathbb{D}\Lambda$, the assertion holds.

$(9)\Rightarrow(8)$ This is obvious.

$(8)\Rightarrow(9)$  By $(8)\Leftrightarrow(1)$, $\Lambda$ is self-injective. Then the unique tilting module in $\Lambda$-$\mod$ is $\mathbb{D}\Lambda$, the assertion holds.
\end{proof}

\begin{remark}\label{3.3}  By Theorem \ref{3.2}, one gets that the fact $\mathbb{D}\Lambda$ is semi-Gorenstein projective $\tau$-rigid is equivalent to the fact $\mathbb{D}\Lambda$ is Gorenstein projective. To prove Tachikawa's Conjecture, it is meaningful to show that the fact $\mathbb{D}\Lambda$ is semi-Gorenstein projective implies the fact $\mathbb{D}\Lambda$ is $\tau$-rigid.
\end{remark}

{\bf Acknowledgement} The authors want to thank Prof. Zhaoyong Huang and Prof. Zhi-Wei Li for useful discussions and suggestions. The authors are supported by the NSFC(Nos. 12171207, 12201211, 12371038).

\end{document}